\documentclass[11pt,a4paper]{amsart}
\usepackage[latin1]{inputenc}
\usepackage[english]{babel}
\usepackage[T1]{fontenc}
\usepackage{amsmath,amssymb}
\usepackage{stmaryrd}
\usepackage{enumerate}
\usepackage{amsthm}
\usepackage{mathrsfs}
\usepackage{geometry}
\usepackage{mathtools}
\usepackage{graphicx}
\usepackage{array}
\usepackage[draft,english]{fixme}
\usepackage[all,line]{xy}
\usepackage[usenames,dvipsnames,svgnames,table]{xcolor}
\usepackage{xfrac}
\mathtoolsset{showonlyrefs}
\newtheorem{thm}{Theorem}[section]
\newtheorem{prop}[thm]{Proposition}
\newtheorem{lemma}[thm]{Lemma}
\theoremstyle{definition}
\newtheorem{defin}[thm]{Definition}
\theoremstyle{remark}
\newtheorem{rem}[thm]{Remark}

\newtheorem{cor}[thm]{Corollary}

\newcommand{\iso}{\simeq}
\newcommand{\isomap}{\stackrel{\sim}{\to}}

\newcommand{\C}{\mathbb{C}}

\newcommand{\Z}{\mathbb{Z}}

\newcommand{\g}{\mathfrak{g}}

\newcommand{\bfr}{\mathfrak{b}}

\newcommand{\flag}{\mathcal{B}}
\newcommand{\gtil}{{\tilde{\mathfrak{g}}}}

\newcommand{\Acal}{\mathcal{A}}

\newcommand{\Ccal}{\mathcal{C}}
\newcommand{\Ocal}{\mathcal{O}}
\newcommand{\Ecal}{\mathcal{E}}
\newcommand{\Fcal}{\mathcal{F}}
\newcommand{\Dcal}{\mathcal{D}}
\newcommand{\Mcal}{\mathcal{M}}
\newcommand{\Ncal}{\mathcal{N}}

\newcommand{\Vcal}{\mathcal{V}}

\newcommand{\Ical}{\mathcal{I}}

\newcommand{\Kcal}{\mathcal{K}}
\newcommand{\Gcal}{\mathcal{G}}
\newcommand{\Ad}{\operatorname{Ad}}

\newcommand{\Sym}{\operatorname{Sym}}

\newcommand{\QCoh}{\operatorname{QCoh}}
\newcommand{\Coh}{\operatorname{Coh}}
\newcommand{\Hom}{\operatorname{Hom}}
\newcommand{\Ext}{\operatorname{Ext}}

\newcommand{\af}{{\operatorname{af}}}

\newcommand{\module}{\operatorname{mod}}

\newcommand{\id}{\operatorname{Id}}

\newcommand{\Dabs}{D^{\operatorname{abs}}}

\newcommand{\odd}{{\operatorname{odd}}}
\newcommand{\even}{{\operatorname{even}}}

\newcommand{\Perf}{\operatorname{Perf}}

\newcommand{\fl}{{\operatorname{fl}}}

\usepackage{tikz-cd}
\usetikzlibrary{decorations.markings}
\tikzset{degil/.style={
        decoration={markings,
            mark= at position 0.5 with {
                \node[transform shape] (tempnode) {$\backslash$};}},
        postaction={decorate}}}

\begin{document}

\title{Braid group actions on matrix factorizations}
\author{Sergey Arkhipov}
\author{Tina Kanstrup}
\address{S.A. Matematisk Institut, Aarhus Universitet, Ny Munkegade, DK-8000 , Aarhus C, Denmark, email: hippie@math.au.dk} 
\address{T.K. Centre for Quantum Geometry of Moduli Spaces, Aarhus Universitet, Ny Munkegade, DK-8000 , Aarhus C, Denmark, email: tina.kanstrup@mail.dk}

\maketitle

\begin{abstract}
Let $X$ be a smooth scheme with an action of a reductive algebraic group $G$ over an algebraically closed field $k$ of characteristic zero. We construct an action of the extended affine Braid group on the $G$-equivariant absolute derived category of matrix factorizations on the Grothendieck variety times $T^*X$ with potential given by the Grothendieck-Springer resolution times the moment map composed with the natural pairing.
\end{abstract}

\section{Introduction}
The present paper is a follow up of the paper \cite{AK}. Recall the setting in that paper.
\subsection{Coherent Hecke category for the derived loop group and affine Hecke category.}
Let $X$ be a regular Noetherian scheme with an action of a reductive algebraic group $G$. 
Fix a Borel subgroup $B\subset G$. In the introduction to \cite{AK} we explained that it is natural to consider the Hecke category for the pair of the derived group schemes
$(L_{\text{top}}(G), L_{\text{top}}(B))$ of derived loops with values in $G$ (resp., in $B$).

 The category
  $\text{QCHecke}(L_{\text{top}}(G), L_{\text{top}}(B))$ 
is a monoidal triangulated category via the convolution functor.  It is expected to act naturally 
on the category of $L_{\text{top}}(B)$-equivariant coherent sheaves on the derived scheme of derived loops with values in the scheme $X$. Formal definitions of the categories above
would rely  heavily on Derived Algebraic Geometry and are unknown to the authors. Thus we use the picture as a source of inspiration and act by analogy with the case of $B$-equivariant 
coherent sheaves on $X$. Our actual construction follows a different direction.

In \cite{BN} Ben-Zvi and Nadler make precise sense of
  $\text{QCHecke}(L_{\text{top}}(G), L_{\text{top}}(B))$ 
in classical algebro-geometric terms. Namely they identify the category with the affine Hecke category of Bezrukavnikov and prove that
$$
  \text{QCHecke}(L_{\text{top}}(G), L_{\text{top}}(B))\isomap
   D^b\Coh^G(St).
  $$
  Here $St$ denotes the Steinberg variety for the group $G$.
  \subsection{Equivariant matrix factorizations.}
  In \cite{AK} we gave a precise definition of the module category discussed above and understood informally as 
  $\Coh^{L_{\text{top}}(B)}(L_{\text{top}}(X))$. Namely, we considered the moment map $$h: T^*X\times Lie(B)\to \mathbb{A}^1$$ and the corresponding equivariant derived category of
  matrix factorizations 
  $\text{DMF}^B(T^*X\times Lie(B),h)$.
  
  In our paper, we construct a monoidal action of the category $D^b\Coh^G(St)$ on this category. For technical reasons it is easier to work with geometric objects (coherent sheaves etc) 
  equivariant with respect to the reductive group $G$. 
  
  We replace the category $\text{DMF}^B(T^*X\times Lie(B),h)$ by the equivalent category of $G$-equivariant matrix factorizations on 
  $$T^*X\times
   \frac{Lie(B)\times G}{B}$$
    Notice that 
   $\tilde{\mathfrak{g}}:
   =\frac{Lie(B)\times G}{B}$ is one of the constructions of the Grothendieck variety for $G$. The potential for the category of matrix factorizations is given by the composition
   
  $$
  T^*X\times
     \tilde{\mathfrak{g}}\to \mathfrak{g}^*\times \mathfrak{g}\to \mathbb{A}^1.
     $$
     Here the first map is a product of the canonical projection with the moment map for $G$-action. The second map is the canonical bilinear pairing.

The main result of the paper is the following statement.

\begin{thm} \label{MThm}
There exists a natural monoidal action of the affine Hecke category $D^b\Coh^G(St)$ on the derived category of equivariant matrix factorizations
$\operatorname{DMF}^G(T^*X\times\tilde{\mathfrak{g}},h)$.
\end{thm}

\subsection{The structure of the paper.}

In section 2 we recall the construction by Bezrukavnikov and Riche \cite{BR} of a categorical action of the affine braid group on the derived category of the dg-category of $G \times \mathbb{G}_m$-equivariant coherent sheaves on $\gtil$, $\Dcal^b \Coh^{G \times \mathbb{G}_m}(\gtil)$. The action is constructed by defining a monoidal action of another category on $\Dcal^b \Coh^{G \times \mathbb{G}_m}(\gtil)$ and then identifying elements in this category whose convolution satisfy affine braid group relations.

In section 3 we recall the basic definitions of matrix factorizations and their absolute derived category.

In section 4 we elaborate on the setup mentioned in section 12.

In section 5 we introduce the functors of derived push-forward, pull-back and tensor product for equivariant matrix factorizations. We prove some relations between these functors which are needed in the subsequent sections.

In section 6 we define a convolution product. Let $X$ and $Y$ be smooth $G$-schemes and $V$ a vector space. Assume that we have equivariant morphisms $\mu : X \to V^*$ and $\nu: Y \to V$. This defines two functions
\begin{gather}
w : Y \times X \to k, \qquad (y,x) \mapsto \mu(x)(\nu(y)),\\
h : Y \times Y \times V^* \to k, \qquad (y_1,y_2,v) \mapsto v(\nu(y_1)-\nu(y_2)).
\end{gather}
Using a convolution, which is a modified analogy of the convolution product used in \cite{BR}, we define a monoidal structure on the category $\Dabs(\QCoh^G(Y \times Y \times V^*),h)$. We also define a monoidal action of this category on the category $\Dabs(\QCoh^G(Y \times X),w)$. Finally, we prove that imposing some support condition on $\Dabs(\QCoh^G(Y \times Y \times V^*),h)$ this induces a monoidal action on $\Dabs(\Coh^G(Y \times X),w)$.

In section 7 we finish the proof of \ref{MThm} by constructing a monoidal functor from a subcategory of the monoidal category used by Bezrukavnikov and Riche to $\Dabs(\QCoh^G(Y \times Y \times V^*),h)$. The subcategory contains the generators for the $B_\af$-action and the image of these under the constructed functor lies in the subcategory of $\Dabs(\QCoh^G(Y \times Y \times V^*),h)$ preserving $\Dabs(\Coh^G(Y \times X),w)$. Thus, convolution with the images of the generators produces a $B_\af$-action on $\Dabs(\Coh^G(Y \times X,w)$.

\subsection{Acknowledgments}
We would like to thank V. Baranovsky, R. Bezrukavnikov, A. I. Efimov and A. Polishchuk and S. Riche for helpful comments.

Both authors' research was supported in part by center of excellence grants "Centre for Quantum Geometry of Moduli Spaces" and by FNU grant "Algebraic Groups and Applications".

\section{The braid group action of Bezrukavnikov and Riche}\label{BraidGrpAct}

In this section we recall the construction of an action of the (extended) affine braid group by Bezrukavnikov and Riche in \cite{BR} and \cite{Ric}. Let $G$ be a reductive algebraic group over an algebraically closed field $k$. Their proof works both in characteristic zero and when the characteristic is bigger than the Coxeter number. Fix a maximal torus $T$ and Borel subgroup $B$ containing it. Recall that the extended affine braid group has the following presentation.

\begin{thm}\cite[Thm 1.1.3]{Ric}
The extended affine braid group, $B_\af$, admits a presentation with generators $\{T_\alpha \mid \alpha \in \Pi \} \cup \{\theta_x \mid x \in \mathbb{X}\}$ and relations:
\begin{enumerate}
\item $T_\alpha T_\beta T_\alpha \cdots = T_\beta T_\alpha T_\beta \cdots $ with $m(\alpha, \beta)$ factors on each side.
\item $\theta_x \theta_y=\theta_{x+y}$.
\item $T_\alpha \theta_x =\theta_x T_\alpha$ if $\langle x, \alpha \rangle=0$, i.e. $s_\alpha(x)=x$.
\item $\theta_x = T_\alpha \theta_{x-\alpha} T_\alpha$ if $\langle x, \alpha \rangle=1$, i.e. $s_\alpha(x)=x-\alpha$.
\end{enumerate}
\end{thm}

We sketch their construction. Let $X$ and $Y$ be $G$-varieties. Denote the projections $X \times Y \to X$ and $X \times Y \to Y$ by $p_X$ and $p_Y$ respectively. These projections are not assumed to be proper, so push-forward might not take coherent sheaves to coherent sheaves. This problem is fixed by introducing the following full subcategory
\[ D^b_{\text{prop}}(\Coh(X \times Y)) \subset D^b(\Coh(X \times Y)) \]
in the following way. An object in $D^b(\Coh(X \times Y))$ belongs to $D^b_{\text{prop}}(\Coh(X \times Y))$ if its cohomology sheaves are topologically supported on a closed subscheme $Z \subset X \times Y$ such that the restrictions of $p_X$ and $p_Y$ to $Z$ are proper. They define a convolution product
\begin{gather}
* : D^b_{\text{prop}}(\Coh^G(Y \times Z)) \times D^b_{\text{prop}}(\Coh^G(X \times Y)) \to D^b_{\text{prop}}(\Coh^G(X \times Z)),\\
\Fcal * \Gcal :=R p_{X,Z*}(p_{X,Y}^* \Fcal \otimes^L_{X \times Y \times Z} p_{Y,Z}^* \Gcal),
\end{gather}
where $p_{X,Y}, p_{Y,Z}$ and $p_{X,Z}$ are the projections from $X \times Y \times Z$ to the listed factors. Any $\Fcal \in D^b_{\text{prop}}(\Coh^G(X \times Y)$) defines a functor
\begin{gather}
F^\Fcal_{X \to Y} : D^b(\Coh^G(X)) \to D^b(\Coh^G(X)),\\
\Mcal \mapsto R p_{Y*}(\Fcal \otimes^L_{X \times Y} p_X^*\Mcal).
\end{gather}

\begin{lemma}\cite[Lemma 1.2.1]{Ric}
Let $\Fcal \in D^b_{\text{prop}}(\Coh^G(X \times Y))$ and $\Gcal \in D^b_{\text{prop}}(\Coh^G(Y \times Z))$. Then
\[ F^\Gcal_{Y \to Z} \circ F^\Fcal_{X \to Y} \iso F^{\Gcal * \Fcal}_{X \to Z}. \]
\end{lemma}

The categorical action of $B_\af$ to be constructed will be a weak action. Recall that a weak action of a group $A$ on a category $\Ccal$ is a group morphism from $A$ to the group of isomorphism classes of auto-equivalences of the category $\Ccal$. Note that no compatibility conditions are imposed on the morphisms. In particular, to construct an action of $B_\af$ on $D^b(\Coh^G(X))$ it suffices to find objects in $(D^b_{\text{prop}}(\Coh^G(X \times X)),*)$, whose convolution with each other satisfy affine braid group relations.

In this construction the variety is going to be the Grothendieck variety $\gtil$. Recall that $\gtil$ is smooth and that $\nu : \gtil \to \g$ is proper. It follows that the base change morphisms $\gtil \times_\g \gtil \to \gtil$ are proper, so we have a full monoidal subcategory
\[ D^b_{\gtil \times_\g \gtil}(\Coh^G(\gtil \times \gtil)) \subset D^b_{\text{prop}}(\Coh^G(\gtil \times \gtil)), \]
whose objects are topologically supported on $\gtil \times_\g \gtil$. It is a monoidal category with $*$. We call this monoidal category the affine Hecke category
\[ \text{Hecke}_{\text{af}}(G,B):=(D^b_{\gtil \times_\g \gtil} (\Coh^G(\gtil \times \gtil)),*). \]
Consider the composition
\[ \gtil \times_\g \gtil \hookrightarrow \gtil \times \gtil \twoheadrightarrow \flag \times \flag. \]
For $w \in W$ we denote by $Z_w$ the closure of the inverse image of the orbit of the point $(B/B, w^{-1}B/B)$ for the diagonal action of $G$. For $x \in \mathbb{X}$ we have the canonical line bundle $\Ocal_\flag(x)$. We define $\Ocal_{\Delta(\gtil)}(x)$ to be the pull-back of $\Ocal_\flag(x)$ along the projection of the diagonal $\Delta(\gtil) \to \flag$. We can now state the main result of Bezrukavnikov and Riche, which can be seen as a categorification of the result about representations of the Weyl group in chapter 2.

\begin{thm}\cite[Theorem 1.3.2]{BR} \label{BRBraidAct}
There is a categorical $B_\af$-action on $D^b(\Coh^G(\gtil))$ in which $T_{\alpha_i}$ acts by convolution with $\Ocal_{Z_{s_{\alpha_i}}}$ and $\theta_x$ acts by convolution with $\Ocal_{\Delta(\gtil)}(x)$.
\end{thm}

Moreover, they prove that for a reduced expression $w=s_{\alpha_{i_1}} \cdots s_{\alpha_{i_n}}$
\[ \Ocal_{Z_{s_{\alpha_{i_1}}}} * \cdots * \Ocal_{Z_{s_{\alpha_{i_n}}}} \iso \Ocal_{Z_w}. \]

In \cite[Section 4 and 5]{BR} Bezrukavnikov and Riche also lift this construction to the DG-category setting. Consider two $G \times \mathbb{G}_m$-equivariant dg-algebras $X=(X_0,\Acal_X)$ and $Y=(Y_0,\Acal_Y)$ and an equivariant DG-morphism of sheaves of DG-algebras $f: X \to Y$. By choosing resolutions (in the appropriate sense) one can define a derived fiber product $X \times_Y^R X$ up to quasi-isomorphism. This is sufficient to get a well-defined category
\[ \Kcal_{X,Y} :=\Dcal \QCoh^{G \times \mathbb{G}_m}(X \times_Y^R X). \]

\begin{rem}
Assume that $X$ and $Y$ are ordinary schemes (i.e. $\Acal_X=\Ocal_X$ and $\Acal_Y=\Ocal_Y$) satisfying 
\[\text{Tor}_{\neq 0}^{f^{-1}\Ocal_Y}(\Ocal_X,\Ocal_X)=0.\]
Then the dg-scheme $X \times_Y^R X$ is just the ordinary fiber product $X_0 \times_{Y_0} X_0$ and
\[\Kcal_{X,Y} \iso \Dcal \QCoh^{G \times \mathbb{G}_m}(X_0 \times_{Y_0} X_0).\] 
This is the case when $X=\gtil$, $Y=\g$ and $f$ is the Grothendieck-Springer resoluion (see \cite[Section 1.2.2]{Bez}).
\end{rem}

The category $\Kcal_{X,Y}$ is monoidal. When $f$ is smooth (any quasi-projective morphism can be replaced by a smooth morphism using a trick explained in \cite[Section 3.7]{BR}) the derived fiber product is the ordinary fiber product of dg-schemes $X \times_Y X$ and the convolution product is defined in the following way. Let $q_{ij} : X_0 \times_{Y_0} X_0 \times_{Y_0} X_0 \to X_0 \times_{Y_0} X_0$ be the projection to the $(i,j)$-th factor. Consider the dg-schme
\[ Z_{ij} :=(X_0 \times_{Y_0} X_0 \times_{Y_0} X_0, q_{ij}^* \Acal_{X \times_Y X}.) \]
It has a natural morphism of dg-schemes $p_{ij} : Z_{ij} \to X \times_Y X$. Let $q_2: X_0 \times_{Y_0} X_0 \times_{Y_0} X_0 \to X_0$ be projection to the second factor, and consider the sheaf of dg-algebras $q_2^* \Acal_X$ on $X_0 \times_{Y_0} X_0 \times_{Y_0} X_0,$. Then there exist a derived tensor product
\[\otimes_{q_2^* \Acal_X}^L :\Dcal \QCoh^{G \times \mathbb{G}_m}(Z_{12}) \times \Dcal \QCoh^{G \times \mathbb{G}_m}(Z_{23}) \to \Dcal \QCoh^{G \times \mathbb{G}_m}(Z_{13})\]
Using this they define the convolution product $* : \Kcal_{X,Y} \times \Kcal_{X,Y} \to \Kcal_{X,Y}$
\[ \Mcal * \Ncal :=R p_{13*}(Lp_{12}^*\Ncal \otimes_{q^* \Acal_X}^L L p_{23}^* \Mcal). \]
Let $p_1,p_2 : X \times_Y X \to X$ be the two projections. There is a monoidal action of $\Kcal_{X,Y}$ on $\Dcal \QCoh^{G \times \mathbb{G}_m}(X)$.
\begin{gather}
\Kcal_{X,Y} \times  \Dcal \QCoh^{G \times \mathbb{G}_m}(X) \to \Dcal \QCoh^{G \times \mathbb{G}_m}(X),\\
\Mcal * \Ncal :=R p_{2*}(\Mcal \otimes^L_{X \times_Y^R X} Lp_1^* \Ncal).
\end{gather}

\begin{defin}
Let $\Kcal^{\Coh}_{X,Y}$ be the full subcategory of $\Kcal_{X,Y}$ whose objects are complexes with only finitely many non-zero cohomology sheaves, each of which is a coherent sheaf on $X \times_Y X$.
\end{defin}

\begin{prop}\cite[Prop. 4.2.1]{BR}
Assume that $X$ and $Y$ are ordinary schemes and that $f$ is proper. Then $\Kcal^{\Coh}_{X,Y}$ is a monoidal category with the restricted convolution product and the action of $\Kcal_{X,Y}^{\Coh}$ on $\Dcal \QCoh^{G \times \mathbb{G}_m}(X)$ preserves the full subcategory $\Dcal^b \Coh^{G \times \mathbb{G}_m}(X)$.
\end{prop}

In the setting of the proposition there is a "direct image under closed embedding" functor
\[\Kcal_{X,Y}^{\Coh} \to \Dcal^b \Coh_{X \times_Y X}(X \times X)\]
This functor is monoidal as remarked in \cite[Section 4.1]{MR3}. The generators of the braid group action from theorem \ref{BRBraidAct} are all schemes on $\gtil \times_\g \gtil$ so they can naturally be considered as objects in $\Kcal_{\gtil,\g}^{\Coh}$.

\begin{thm}\label{DGBraidAct}
There is a (weak) categorical action of $B_\af$ on $\Kcal_{\gtil,\g}^{\Coh}$.
\end{thm}

\section{Matrix factorizations}

Let $X$ be a separated, Noetherian $G$-scheme with enough $G$-equivariant vector bundlesbundles, i.e. every coherent sheaf on X is the quotient sheaf of a $G$-equivariant locally free sheaf of finite rank.  Let $w : X \to k$ be a regular $G$-invariant function.

\begin{defin}
\begin{enumerate}
\item A matrix factorization with potential $w$ is a quadruple 
\[(\Mcal^{-1}, \Mcal^0, d_{-1},d_0),\] 
where $\Mcal^{-1}$ and $\Mcal^0$ are coherent sheaves on $X$ with morphisms $d_{-1} : \Mcal^{-1} \to \Mcal^0$ and $d_0: \Mcal^0 \to \Mcal^{-1}$ satisfying that $d_0 \circ d_{-1}$ and $d_{-1} \circ d_0$ are multiplication by $w \in \Ocal(X)$.
\item A matrix factorization is $G$-equivariant if the sheaves $\Mcal^{-1}$ and $\Mcal^0$ are $G$-equivariant and the morphisms $d_{-1},$ and $d_0$ are $G$-equivariant.
\end{enumerate}
\end{defin}

The category of matrix factorizations on $X$ with potential $w$ has a DG-category structure with Hom complex
\[ \Hom^n(\Mcal, \Ncal):=\Hom_{\Coh^G(X)}(\Mcal^{-1}, \Ncal^{-1+n \text{ mod } 2}) \oplus \Hom_{\Coh^G(X)}(\Mcal^{0}, \Ncal^{n \text{ mod } 2}), \]
with differential
\[ d(f)_i(m):=d_{N}(f_i(m))-(-1)^{i} f_i(d_{M}(m)). \]
Indeed,
\begin{align}
d^2(f)_i(m) &=d^2_N f_i(m)-(-1)^{i+1} d_N f_i d_M(m)-(-1)^{i} d_N f_i d_M(m)-f_i d_M^2(m)\\
&=wf(m)-f(wm)\\
&=0.
\end{align}
We denote the homotopy category by $H^0(\Coh^G(X),w)$. The corresponding category where we only require the terms in the matrix factorization to lie in $\QCoh^G(X)$ is denoted by $H^0(\QCoh^G(X),w)$. The DG-categories allow shifts and cones so the homotopy categories are triangulated. Let 
\[\Mcal^\bullet=\dots \to \Mcal_i \stackrel{g_i}{\to} \Mcal_{i+1} \to \dots\]
be a complex $(g_i \circ g_{i-1}=0)$ of matrix factorizations with potential $w$. The total object of $\Mcal^\bullet$ is the matrix factorization
\begin{gather}
\text{Tot}(\Mcal^\bullet)=\Bigl(\bigoplus_{2i} \Mcal_{2i}^{-1} \oplus \bigoplus_{2i+1} \Mcal_{2i+1}^0, \bigoplus_{2i} \Mcal_{2i}^0 \oplus \bigoplus_{2i+1} \Mcal_{2i+1}^{-1}, d_{-1}, d_0\Bigr)\\
d=\sum_n d_{\Mcal_n} +(-1)^{\text{vertical degree}} g_n
\end{gather}
There is a notion of exotic derived category due to Positselski
\begin{defin}
\begin{enumerate}
\item Objects in $H^0(\Coh^G(X),w)$ are called equivariant absolutely acyclic if they belong to the minimal thick subcategory of $H^0(\Coh^G(X),w)$ containing all total complexes of short exact sequences in $(\Coh^G(X),w)$.
\item The equivariant absolute derived category $\Dabs(\Coh^G(X),w)$ is the quotient category of $H^0(\Coh^G(X),w)$ by the thick subcategory of equivariant absolutely acyclic matrix factorizations.
\item The equivariant absolute derived category $\Dabs(\QCoh^G(X),w)$ is the quotient category of $H^0(\QCoh^G(X),w)$ by the minimal thick subcategory containing all total complexes of short exact sequences in $(\QCoh^G(X),w)$.
\item Let $\Ecal$ be an exact subcategory in $\QCoh^G(X)$. Denote the full subcategory of $(\QCoh^G(X),w)$ whose objects have terms in $\Ecal$ by $(\QCoh^G(X)_\Ecal,w)$. The relative absolute derived category $\Dabs(\QCoh^G(X)_\Ecal,w)$ is the quotient of $H^0(\QCoh^G(X)_\Ecal,w)$ by the minimal thick subcategory containing the total objects of all short exact sequences whose terms are in $\Ecal$. The category $\Dabs(\Coh^G(X)_\Ecal,w)$ is defined similarly.
\end{enumerate}
\end{defin}

\begin{lemma}\cite[Rem. 1.3]{EP} \label{local}
Absolute acyclicity is a local notion, i.e. to check that $\Mcal$ is acyclic it is enough to show that $\Mcal$ restricted to each $U_\alpha$ is acyclic, where $\{U_\alpha\}$ is a finite affine open cover of $X$.
\end{lemma}

\begin{prop}  \label{FullyFaithful}
Let $G$ be a reductive algebraic group acting on a smooth scheme $X$. Then the functor $\Dabs(\Coh^G(X),w) \to \Dabs(\QCoh^G(X),w)$ induced by the inclusion $(\Coh^G(X),w) \hookrightarrow (\QCoh^G(X),w)$ is fully faithful.
\end{prop}
\begin{proof}
Without equivariance this is \cite[Prop. 1.5(c)]{EP}. The proof extends to the equivariant setting.
\end{proof}

\section{Braid group action on matrix factorizations coming from Hamiltonian actions} \label{BraidActMatFact}

Let $G$ be a reductive algebraic group over an algebraically closed field of characteristic zero and $X$ a smooth $G$-scheme. Then we have a Hamiltonian action of $G$ on $T^*X$ with moment map $\mu : T^*X \to \g^*$ which is $G$-equivariant with the $G$-action on $\g^*$ being the coadjoint action. Combining this with the Grothendieck-Springer resolution we get a function
\[ w: \gtil \times T^*X \stackrel{\nu \times \mu}{\longrightarrow} \g \times \g^* \stackrel{\langle \:,\: \rangle}{\longrightarrow} \C.   \]
We check that this function is $G$-invariant
\begin{align}
w(g\cdot x,g\cdot y,g\cdot z)&=\langle \Ad(g)x, \mu(g \cdot z) \rangle=\langle \Ad(g)x, g\cdot \mu(z) \rangle\\
&=\langle \Ad(g^{-1})\Ad(g)x,\mu(z) \rangle=\langle x,\mu(z) \rangle\\
&=w(x,y,z).
\end{align}
Thus, this defines a potential for equivariant matrix factorizations on $\gtil \times T^*X$. The theorem we are going to prove is the following.

\begin{thm}[Main theorem] \label{MainThm}
There is a $B_\af$-action on $\Dabs(\Coh^G(\gtil \times T^*X),w)$.
\end{thm}

The idea of the proof is to construct a monoidal category of matrix factorizations and a monoidal action of this category on $\Dabs(\Coh^G(\gtil \times T^*X),w)$. Then we construct a monoidal functor from a subcategory of the category from Bezrukavnikov and Riche's theorem \ref{DGBraidAct}, containing the generators of the braid group action, to this monoidal category.

Under the isomorphism $\frac{G \times \bfr}{B} \isomap \gtil$ given by $(g,x) \mapsto (\Ad(g)x, gB/B)$ the Grothendieck-Springer resolution becomes the map $(g,x) \mapsto \Ad(g)x$. Thus, we have an equivalence
\[ \Dabs(\Coh^G(\gtil \times T^*X),w)\iso \Dabs(\Coh^{G \times B}(G \times \bfr \times T^*X),w'), \]
where $w'$ is the morphism $(g,x) \mapsto \Ad(g)x$. The $G$-action on $\frac{G \times \bfr}{B}$ is multiplication on the $G$-factor. We also have the equivalence
\[ \Dabs(\Coh^{G \times B}(G \times \bfr \times T^*X),w')\iso \Dabs\Bigl(\Coh^{B}\Bigl(\frac{G \times \bfr \times T^*X}{G}\Bigr),\bar{w'}\Bigr)\]
Under the isomorphism $\frac{G \times \bfr \times T^*X}{G}\iso \bfr \times T^*X$ the potential $\bar{w'}$ becomes the map
\[ \bfr \times T^*X \to \C, \qquad (b,x) \mapsto \langle b , \mu(x) \rangle. \]
Thus, the theorem gives a $B_\af$-action on the equivalent categories
\begin{align}
 \Dabs(\Coh^G(\gtil \times T^*X),w)&\iso \Dabs(\Coh^B(\bfr \times T^*X),\bar{w'}) \\
& \iso \Coh^{L_{\text{top}}(B)}(L_{\text{top}}(X)).
\end{align}

\section{Functors on matrix factorizations}
Let $\Mcal, \Ncal \in (\QCoh^G(X),w)$. The formula for tensor products of complexes gives a tensor product on matrix factorizations
\begin{gather} 
\Mcal \otimes_X \Ncal:=(\Mcal^{-1} \otimes_X \Ncal^0 \oplus \Mcal^0 \otimes_X \Ncal^{-1}, \Mcal^{-1} \otimes_X \Ncal^{-1} \oplus \Mcal^0 \otimes_X \Ncal^0, d_{-1},d_0)\\
d_i(m \otimes n):=d_i^\Mcal(m) \otimes b +(-1)^i a \otimes d_i^\Ncal(b)
\end{gather}
A matrix factorization is $\Mcal \in (\QCoh^G(X),w)$ is flat if the functor $- \otimes_X \Mcal$ is exact on the Abelian category of matrix factorizations. This is the case if the terms in $\Mcal$ are flat over $\Ocal_X$. The full subcategory of flat (resp. locally free) matrix factorizations is denoted by $(\QCoh^G(X)_{\text{fl}},w)$ (resp. $(\QCoh^G(X)_{\text{lf}},w)$). Since $\otimes_X$ takes short exact sequences in $(\QCoh^G(X)_\fl,w)$ to short exact sequences it takes absolutely acyclic modules to absolutely acyclic modules. Hence, it induces a functor
\[ \otimes^L_{X} :\Dabs(\QCoh^G(X),w_1) \times \Dabs(\QCoh^G(X)_\fl,w_2) \to \Dabs(\QCoh^G(X),w_1+w_2). \]
Let $f:X \to Y$ be a $G$-equivariant morphism. Then we can define pull-back.
\begin{gather}
f^* : (\QCoh^G(Y),w \circ f) \to (\QCoh^G(X,w)),\\
\Mcal \mapsto (f^*\Mcal^{-1}, f^* \Mcal^0, f^* d_{-1},f^* d_0).
\end{gather}
By the same argument as above it induces a functor
\[ Lf^{*} : \Dabs(\QCoh^G(X)_\fl,w) \to \Dabs(\QCoh^G(X)_\fl, w \circ f). \]
Finally, we have push-forward
\begin{gather}
f_* : (\QCoh^G(X),w \circ f) \to (\QCoh^G(Y),w), \qquad\\
\Mcal \mapsto (f_* \Mcal^{-1}, f_* \Mcal^0,f_* d_{-1}, f_*d_0)
\end{gather}
We denote the full subcategory of $(\QCoh^G(X),w)$ whose terms are injective (resp. finite injective dimension) by $(\QCoh^G(X)_{\text{inj}},w)$ (resp. $(\QCoh^G(X)_{\text{fid}},w)$). Then we have a derived functor
\[ R f_* : \Dabs(\QCoh^G(Y)_{\text{inj}},w \circ f) \to \Dabs(\QCoh^G(X),w) \]

\begin{prop}\label{LocFreeInj}
Let $G$ be a reductive algebraic group acting on a smooth scheme $X$.
\begin{enumerate}
\item The natural functor 
\[H^0(\QCoh^G(X)_{\operatorname{inj}},h) \to \Dabs(\QCoh^G(X),h)\] 
is an equivalence of triangulated categories.
\item The natural functor 
\[\Dabs(\QCoh^G(X)_{\operatorname{lf}},w) \to \Dabs(\QCoh^G(X),w)\] 
is an equivalence of triangulated categories.
\end{enumerate}
\end{prop}
\begin{proof}
(1) By \cite[Lemma 1.7(a)]{EP} there is an equivalence of triangulated categories induced by the inclusion
\[H^0(\QCoh(X)_{\operatorname{inj}},w) \iso \Dabs(\QCoh(X)_{\operatorname{fid}},w).\]
The proof extends to the equivariant setting. For smooth schemes all equivariant quasi-coherent sheaves have finite equivariant injective dimension so the result follows.

(2) \cite[Cor. 2.4(b)+rem.]{EP} states that when $X$ is a regular separated Noetherian scheme of finite Krull dimension then the natural functor $\Dabs(\QCoh(X)_\text{lf},w) \to \Dabs(\QCoh(X),w)$ is an equivalence of triangulated categories. The proof extends to the equivariant setting when $X$ has finite $G$-equivariant locally free dimension. This is satisfied when $X$ is smooth.
\end{proof}

\begin{cor}
Assume that $X$ and $Y$ are smooth and let $f : X \to Y$ be a $G$-equivariant morphism.
\begin{enumerate}
\item There is a tensor product
\[ \otimes_X^L : \Dabs(\QCoh^G(X),w_1) \times \Dabs(\QCoh^G(X), w_2) \to \Dabs(\QCoh^G(X), w_1 +w_2)\]
\item There is a pull-back
\[Lf^{*} : \Dabs(\QCoh^G(X),w) \to \Dabs(\QCoh^G(Y),w \circ f).\]
\item There is a push-forward
\[ R f_* : \Dabs(\QCoh^G(Y),w \circ f) \to \Dabs(\QCoh^G(X),w) \]
\end{enumerate}
\end{cor}

From now on we will always assume that our schemes are smooth.

\begin{prop}
Let $f : X \to Y$ be an equivariant morphism of smooth $G$-schemes. Then $Rf_*$ is right adjoint to $Lf^*$.
\end{prop}
\begin{proof}
In the non-equivariant setting this follows from \cite{EP} Prop. 1.9 and Cor. 2.3(b)+(f). The proof also works in the equivariant case.
\end{proof}

\begin{lemma}
Let $f: X \to Y$ and $g: Y \to Z$ be morphisms of smooth schemes. Then
\[ L(g \circ f)^* \iso Lf^* \circ Lg^* \qquad \text{and} \qquad R(g \circ f)_* \iso R g_* \circ Rf_*.\]
\end{lemma}
\begin{proof}
The first part follows from the fact that pull-back takes flat to flat. The second part follows from the first by adjunction.
\end{proof}

\begin{lemma}
Let $f: X \to Y$ be an equivariant morphism of smooth schemes. There is an isomorphism of functors
\[ Lg^*( -) \otimes^L_X Lg^*(-) \iso Lg^*(- \otimes^L_X -). \]
\end{lemma}
\begin{proof}
The pull-back takes flat modules to flat modules and the tensor product is a functor $\otimes_Y^L : \Dabs(\QCoh(Y)_\fl,w_1) \times \Dabs(\QCoh(Y)_\fl,w_2) \to \Dabs(\QCoh(Y)_\fl, \linebreak[1] w_1+w_2)$. Thus, on $ \Dabs(\QCoh(Y)_\fl,w_1) \times \Dabs(\QCoh(Y)_\fl,w_2)$.
\begin{gather} 
Lg^*( -) \otimes^L_X Lg^*(-) \iso L(g^*(-) \otimes_X g^*(-))\\
Lg^*(- \otimes^L_X -) \iso L(g^*(- \otimes_Y -)).
\end{gather}
Thus, it is enough to show that $g^*(-) \otimes_X g^*(-) \iso g^*(- \otimes_Y -)$
\begin{align}
g^*(\Fcal_1,\Fcal_2) &\otimes_X g^*(\Gcal_1,\Gcal_2)\\
& =(g^* \Fcal_1 \otimes_X g^*\Gcal_2 \oplus g^* \Fcal_2 \otimes_X g^*\Gcal_1, g^* \Fcal_1 \otimes_X g^*\Gcal_1 \oplus g^* \Fcal_2 \otimes_X g^*\Gcal_2)\\
& \iso (g^*(\Fcal_1 \otimes_Y \Gcal_2) \oplus g^*(\Fcal_2 \otimes_Y \Gcal_1), g^*(\Fcal_1 \otimes_Y \Gcal_1) \oplus g^*(\Fcal_2 \otimes_Y \Gcal_2))\\
&=g^*((\Fcal_1,\Fcal_2) \otimes_Y (\Gcal_1, \Gcal_2)).
\end{align}
Clearly, the differentials also match. The result now follows from proposition \ref{LocFreeInj}.
\end{proof}

\begin{prop}[Projection formula]\label{ProjFormula}
Let $g: X \to Y$ be an equivariant morphism of smooth $G$-schemes. Then there are isomorphisms of functors
\[ Rg_*(Lg^*(-) \otimes^L_X -) \iso - \otimes_Y^L Rg_*(-), \qquad Rg_*(- \otimes_X^L Lg^*(-)) \iso Rg_*(-) \otimes_Y^L -.\]
\end{prop}
\begin{proof}
The proof is similar to the proof for the usual derived category of quasi-coherent sheaves (see for example \cite[Lemma 20.8.2 or 21.37.1]{Stacks}). We only prove the first formula since the other is similar. Like for quasi-coherent sheaves we can use the adjunction to construct a morphism. 
\[ - \otimes_Y^L Rg_*(-) \to Rg_*(Lg^*(-) \otimes^L_X -).\]
Indeed, the adjunction morphism $Lg^* Rg_* \to \id$ induces a morphism
\[ Lg^*(- \otimes_Y^L Rg_*(-)) \iso Lg^*(-) \otimes_X^L Lg^* Rg_*(-) \to Lg^*(-) \otimes^L_X -. \]
We obtain the desired morphism from the above morphism by adjunction.

If $\Fcal$ is flat and $\Ical$ is injective then $\Fcal \otimes_X \Ical$ is injective by \cite[Lemma 2.5]{EP}. Thus, when restricting to $\Dabs(\QCoh(Y)_\fl,w_Y) \times \Dabs(\QCoh(X)_{\text{inj}},w_X)$ we have
\begin{gather} 
Rg_*(Lg^*(-) \otimes^L_X -)=g_*(g^*(-) \otimes_X -),\\
- \otimes_Y^L Rg_*(-) = - \otimes_Y g_*(-).
\end{gather}
Thus, we only need to show that $g_*(g^*\Fcal \otimes_X \Gcal)\iso  \Fcal \otimes_Y g_*(\Gcal)$. By \cite[Cor. 2.3(h)]{EP} the inclusion of $\Dabs(\QCoh(Y)_{\text{lf}},w_Y)$ into $\Dabs(\QCoh(Y)_\fl,w_Y)$ is an equivalence of categories, so we may assume that $\Fcal \in \Dabs(\QCoh(Y)_\text{lf},w_Y)$. By Lemma \ref{local} proving that we have an isomorphism can be done locally. Hence, we may assume that $\Fcal=(\Ocal_Y^{\otimes n}, \Ocal_Y^{\otimes m})$. Write $\Gcal=(\Gcal_1,\Gcal_2)$. Then
\begin{align}
g_*(g^*(\Ocal_Y^{\otimes n}, \Ocal_Y^{\otimes m}) \otimes_X (\Gcal_1,\Gcal_2))&=g_*((\Ocal_X^{\otimes n}, \Ocal_X^{\otimes m}) \otimes_X (\Gcal_1,\Gcal_2))\\
&=g_*(\Gcal_2^{\otimes n} \oplus \Gcal_1^{\otimes m}, \Gcal_1^{\otimes n} \oplus \Gcal_2^{\otimes m})\\
&=(g_*\Gcal_2^{\otimes n} \oplus g_*\Gcal_1^{\otimes m}, g_*\Gcal_1^{\otimes n} \oplus g_*\Gcal_2^{\otimes m}).
\end{align}
On the other side we have
\begin{align}
(\Ocal_Y^{\otimes n}, \Ocal_Y^{\otimes m}) \otimes_Y g_*(\Gcal_1,\Gcal_2)) &=(g_*\Gcal_2^{\otimes n} \oplus g_*\Gcal_1^{\otimes m}, g_*\Gcal_1^{\otimes n} \oplus g_*\Gcal_2^{\otimes m}).
\end{align}
We check that the differentials match
{\allowdisplaybreaks
\begin{align}
g_*(g^*d_i \otimes_X d_i^\Gcal)(1 \otimes a)&=g_* g^* d_i(1) \otimes a +(-1)^{|1|} 1 \otimes g^* d_i^\Gcal(a)\\
&=(d_i(1)\otimes 1) \otimes a +(-1)^{|1|}1 \otimes 1 \otimes d_i^\Gcal(a)\\
&\iso d_i(1)a +(-1)^{|1|}d_i^\Gcal(a)
\end{align}}
On the other side we have
\begin{align}
(d_i \otimes_Y g_* d_i^\Gcal)(1 \otimes a)&=d_i(1)\otimes a+(-1)^{|1|} 1 \otimes g_* d_i^\Gcal(a)\\
&\iso d_i(1)a +(-1)^{|1|}d_i^\Gcal(a).
\end{align}
The result now follows from proposition \ref{LocFreeInj}.
\end{proof}

\begin{prop}\label{BaseChange}
Consider a Cartesian square of equivariant morphisms
\[ \xymatrix{Z \ar[r]^{h} \ar[d]_{f'} & Y \ar[d]^f \\ X \ar[r]_g & W} \]
Assume that either $u$ or $v$ is flat. Then there is a natural isomorphism between the composition of derived functors
\[ Lg^* \circ Rf_* \iso R {f'}_* \circ {Lh}^*. \]
\end{prop}
\begin{proof}
The proof is inspired by the proof of tor-independent base change for quasi-coherent sheaves (see for example \cite[Lemma 35.17.3]{Stacks}). By \cite[Proposition 1.9]{EP} derived push-forward for matrix factorizations is right adjoint to pull-back so we can use the same construction as for coherent sheaves to get a canonical base change morphism
\[ Lg^* R f_* \Mcal \to R {f'}_* Lh^* \Mcal, \qquad \Mcal \in \Dabs(\QCoh(Y,w)).\]
This morphism is the adjoint of the morphism
\[ L{f'}^{*} Lg^* Rf_* \Mcal \iso Lh^* Lf^* R f_* \Mcal \to Lh^* \Mcal, \]
which is induced by the adjunction morphism $Lf^* Rf_* \Mcal \to \Mcal$.

That the base change morphism is an isomorphism can be checked locally by lemma \ref{local}. Hence, we may assume that $S'$ and $S$ are affine so $g_*$ is exact. We claim that it is enough to show that
\begin{align} 
Rg_* Lg^* Rf_* \Mcal \to Rg_* Rf'_* Lh^* \Mcal \iso R f_* Rh_* Lh^* \Mcal \label{1}
\end{align}
is an isomorphism. The reason is that a morphism $\alpha$ is an isomorphism if and only if $R g_* \alpha$ is an isomorphism. Indeed, by exactness $\text{Cone}(R g_* \alpha)=g_* \text{Cone}(\alpha)$ but $g_*$ is just restriction so if $g_* \text{Cone}(\alpha) \iso 0$ it is because $\text{Cone}(\alpha) \iso 0$.

To simplify the notation we write $\overline{\Fcal}$ for the matrix factorization $(0,\Fcal,0,0)$. Notice that 
\[(\Gcal_0,\Gcal_1,d_0,d_1) \otimes_X \overline{\Fcal}=(\Gcal_0 \otimes_X \Fcal, \Gcal_1 \otimes_X \Fcal, d_0,d_1).\]
In particular, $\otimes_X \overline{\Ocal_X}$ is the identity. Using the projection formula we get
\[ Rg_* Lg^{*} \Ncal \iso Rg_*(Lg^* \Ncal \otimes^{L}_X \overline{\Ocal_X})\iso \Ncal \otimes^{L}_W Rg_*  \overline{\Ocal_X}. \]
By base change we get that $h$ is an affine morphism so it is exact and the same formula works for $h$. Thus, \eqref{1} can be rewritten as
\[ R f_* \Mcal \otimes^{L}_W Rg_* \overline{\Ocal_X} \to R f_*(\Mcal \otimes^{L}_Y h_* \overline{\Ocal_Z}). \]
Notice that $h_* \overline{\Ocal_Z}=h_* f^{'*} \overline{\Ocal_X}$. Inserting this we get
\[ R f_* \Mcal \otimes^{L}_W Rg_* \overline{\Ocal_X} \to R f_*(\Mcal \otimes^L_Y f^* g_* \overline{\Ocal_X}) \]
If $f^* g_* \overline{\Ocal_X} \iso  Lf^* g_* \overline{\Ocal_X}$ this is the morphism in the projection formula and we are done. For $f$ flat this is clear. In the case where $g$ is flat we use the projection formula
\begin{align}
Lf^* g_* \overline{\Ocal_X}&=g_* \overline{\Ocal_X} \otimes_W^{L} \overline{\Ocal_Y}\\
& \iso g_*(\overline{\Ocal_X} \otimes_X^L L g^* \overline{\Ocal_X}) \iso g_* Lg^* \overline{\Ocal_X}.
\end{align}
Hence, we also have $f^* g_* \overline{\Ocal_X} \iso  Lf^* g_* \overline{\Ocal_X}$ in the case where $g$ is the flat morphism.
\end{proof}

\section{Convolution}

In this section we define a monoidal action on the category from section \ref{BraidActMatFact}. As in \cite{BR} the categorical action of the affine braid group will come from convolution with certain elements.

\subsection{Definition}

Let $X$ and $Y$ be smooth $G$-schemes and $V$ a $G$-vector space. Let $V^*$ denote the dual vector space with the following $G$-action 
\[ g \cdot f(x)=f(g^{-1}x) \qquad \forall g \in G, f \in \Hom(V,k), x \in V.\] 
Assume that we have equivariant morphisms $\mu : X \to V^*$ and $\nu : Y \to V$. These determine a $G$-invariant section $w \in \Ocal(Y \times X)$ by
\[ w : Y \times X \to k, \qquad (y,x) \mapsto \mu(x)(\nu(y)).\]

\begin{rem}
In the case we are interested in $X=T^*X$, $Y=\gtil$, $V=\g$, $\mu$ is the moment map and $\nu$ is the Grothendieck-Springer resolution.
\end{rem}

We would like construct an action on $\Dabs(\QCoh^G(Y \times X),w)$ similar to the one for coherent sheaves in \cite{BR}. However, if we use the exact same formula then the potentials will not match and we will land in the wrong category. To correct this, we introduce an additional factor of $V^*$
\[ \xymatrix{ & Y \times Y \times X \ar[ld]_{p} \ar[rd]^{p_{23}} \ar[d]^{p_{13}}& \\
Y \times Y \times V^* & Y \times X & Y \times X} \]
where $p:=\id \times \id \times \mu$. Define the potential on $Y \times Y \times V^*$ to be the section given by
\[ h : Y \times Y \times V^* \to k, \qquad (y_1,y_2,g) \mapsto g(\nu(y_1)-\nu(y_2)). \]
Then we have
\begin{align}
(w \circ p_{23}+h \circ p)(y_1,y_2,x) &=\mu(x)(\nu(y_2))+\mu(x)(\nu(y_1)-\nu(y_2))\\
&=\mu(x)(\nu(y_1))=w \circ p_{13}(y_1,y_2,x),
\end{align}
Thus, we can define the action $*$ to be the composition.
\[ \small{\xymatrix{\Dabs(\QCoh^G(Y \times Y \times V^*),h) \times \Dabs(\QCoh^G(Y \times X),w) \ar[d]^{Lp^* \times p_{23}^*}\\
\Dabs(\QCoh^G(Y \times Y \times X),h \circ p) \times \Dabs(\QCoh^G(Y \times Y \times X),w \circ p_{23}) \ar[d]^{\otimes^L_{Y \times Y \times X}}\\
\Dabs(\QCoh^G(Y \times Y \times X),h\circ p+w \circ p_{23}=w \circ p_{13}) \ar[d]^{Rp_{13*}}\\
\Dabs(\QCoh^G(Y \times X),w)}}\]

Now we need to define a monoidal structure on $\Dabs(\QCoh^G(Y \times Y \times V^*),h)$. Consider the projection maps
\[ \xymatrix{ & Y \times Y \times Y \times V^* \ar[ld]_{p_{12}} \ar[rd]^{p_{23}} \ar[d]^{p_{13}}& \\
Y \times Y \times V^* & Y \times Y \times V^* & Y \times Y \times V^*} \]
Notice that
\begin{align}
(h \circ p_{12} +h \circ p_{23})(y_1,y_2,y_3,g) &=g(\nu(y_1)-\nu(y_2))+g(\nu(y_2)-\nu(y_3))\\
&=g(\nu(y_1)-\nu(y_3))=h \circ p_{13}(y_1,y_2,y_3,g).
\end{align}
Thus, we can define the convolution product $*$ as the composition.

\begin{align} 
\begin{tikzcd}
\Dabs(\QCoh^G(Y \times Y \times V^*),h) \times \Dabs(\QCoh^G(Y \times Y \times V^*),h) \arrow{d}{p_{12}^* \times p_{23}^*}\\
\Dabs(\QCoh^G(Y \times Y \times Y \times V^*),h \circ p_{12}) \times \Dabs(\QCoh^G(Y \times Y \times Y \times V^*),h \circ p_{23})\arrow{d}{\otimes^L_{Y \times Y \times Y \times V^*}}\\
\Dabs(\QCoh^G(Y \times Y \times Y \times V^*),h\circ p_{12}+h \circ p_{23}=h \circ p_{13}) \arrow{d}{Rp_{13*}}\\
\Dabs(\QCoh^G(Y \times Y \times V^*),h).
\end{tikzcd}
\end{align}
The proof that the convolution product is associative is similar to the proof of proposition \ref{assoc} below and both are similar to the proof of associativity in chapter 5.

\begin{prop}\label{assoc}
Let $\Mcal_1, \Mcal_2 \in \Dabs(\QCoh^G(Y \times Y \times V^*),h)$ and $\Ncal \in \Dabs(\QCoh^G(Y \times X),w)$. Then
\[ \Mcal_1*(\Mcal_2 * \Ncal) \iso (\Mcal_1 * \Mcal_2)* \Ncal. \]
\end{prop}
\begin{proof}
Consider the following cartesian diagram of projections
\[ \xymatrix{Y \times Y \times Y \times X \ar[r]^{p_{234}} \ar[d]_{p_{124}} & Y \times Y \times X \ar[d]^{p_{13}} \\
Y \times Y \times X \ar[r]_{p_{23}} & Y \times X} \]
Using the flat base change from proposition \ref{BaseChange} and the projection formula from proposition \ref{ProjFormula} we get
\begin{align}
\Mcal_1 * &(\Mcal_2 * \Ncal) \\
& =R p_{13*}(Lp^*\Mcal_1 \otimes_{Y \times Y \times X}^L p_{23}^* Rp_{13*}(Lp^*\Mcal_2 \otimes_{Y \times Y \times X}^L p_{23}^*\Ncal))\\
&\iso R p_{13*}(Lp^*\Mcal_1 \otimes_{Y \times Y \times X}^L Rp_{124*} p_{234}^*(Lp^*\Mcal_2 \otimes_{Y \times Y \times X}^L p_{23}^*\Ncal))\\
& \iso R p_{13*} Rp_{124*}(p_{124}^* Lp^*\Mcal_1 \otimes_{Y \times Y \times Y \times X}^L p_{234}^*(Lp^*\Mcal_2 \otimes_{Y \times Y \times X}^L p_{23}^*\Ncal)) \\
& \iso R p_{13*} Rp_{124*}(p_{124}^* Lp^*\Mcal_1 \otimes_{Y \times Y \times Y \times X}^L p_{234}^* Lp^*\Mcal_2 \otimes_{Y \times Y \times Y \times X}^L  p_{234}^* p_{23}^*\Ncal).
\end{align}
Since $p_{13} \circ p_{124}=p_{13} \circ p_{134}$ and $p_{23} \circ p_{234}=p_{23} \circ p_{134}$ we get
\begin{align}
\Mcal_1 * &(\Mcal_2 * \Ncal) \\
& \iso R p_{13*} Rp_{134*}(p_{124}^* Lp^*\Mcal_1 \otimes_{Y \times Y \times Y \times X}^L p_{234}^*Lp^*\Mcal_2 \otimes_{Y \times Y \times Y \times X}^L  p_{134}^* p_{23}^*\Ncal)\\
& \iso R p_{13*} (Rp_{134*} (p_{124}^*L p^*\Mcal_1 \otimes_{Y \times Y \times Y \times X}^L p_{234}^*Lp^*\Mcal_2) \otimes_{Y \times Y \times X}^L p_{23}^*\Ncal).
\end{align}
Set $p_4:=\id \times \id \times \id \times \mu$. Notice that $p \circ p_{124}=\pi_{12} \circ p_4$ and $p \circ p_{234}=\pi_{23} \circ p_4$.
\begin{align}
\Mcal_1 * &(\Mcal_2 * \Ncal) \\
&\iso R p_{13*}  (Rp_{134*} (Lp_{4}^* \pi_{12}^*\Mcal_1 \otimes_{Y \times Y \times Y \times X}^L Lp_{4}^* \pi_{23}^*\Mcal_2) \otimes_{Y \times Y \times X}^L p_{23}^*\Ncal)\\
& \iso R p_{13*} (Rp_{134*} Lp_{4}^* (\pi_{12}^*\Mcal_1 \otimes_{Y \times Y \times Y \times V^*}^L \pi_{23}^*\Mcal_2) \otimes_{Y \times Y \times X}^L p_{23}^*\Ncal)
\end{align}
Consider the cartesian diagram
\[ \xymatrix{Y \times Y \times Y \times X \ar[r]^{p_{4}} \ar[d]_{p_{134}} & Y \times Y \times Y \times V^* \ar[d]^{\pi_{13}} \\
Y \times Y \times X \ar[r]_{p} & Y \times Y \times V^*} \]
Using flat base change we get the result
\begin{align}
\Mcal_1 * &(\Mcal_2 * \Ncal) \\
& \iso R p_{13*} (Lp^* R\pi_{13*}(\pi_{12}^*\Mcal_1 \otimes_{Y \times Y \times Y \times V^*}^L \pi_{23}^*\Mcal_2) \otimes_{Y \times Y \times X}^L p_{23}^*\Ncal)\\
&=(\Mcal_1 * \Mcal_2) * \Ncal. \qedhere
\end{align}
\end{proof}

\subsection{Restriction to an action on the coherent category}

The category normally referred to as derived equivariant matrix factorizations is the category $\Dabs(\Coh^G(Y \times X),w)$ so we would like  our action to restrict to this category. The derived pull-back and tensor products both restricts to the coherent category. However, this is not the case for the push-forward along a non-proper map. Bezrukavnikov and Riche solved the corresponding problem for coherent sheaves by introducing a support condition. A similar notion exists in our setting. 

\begin{defin}
\begin{enumerate}[(i)]
\item The category-theoretical support of an equivariant coherent sheaf $\Mcal$ on $X$ is the minimal closed subset $T \subset X$ such that $\Mcal|_{X\backslash T}$ is absolutely acyclic in $\Dabs(\Coh^G(X),w)$.
\item For $T$ a closed subset of a scheme $X$ we denote by $\Dabs_T(\Coh^G(X),w)$ the quotient category of the homotopy category of coherent matrix factorizations category-theoretically supported inside $T$ by the thick subcategory of matrix factorizations which are absolutely acyclic in $\Dabs(\Coh^G(X),w)$.
\end{enumerate}
\end{defin}

The category $\Dabs_T(\Coh^G(X),w)$ is a full subcategory in $\Dabs(\Coh^G(X),w)$. In the non-equivariant setting  this is \cite[Prop. 1.10(d)]{EP} and the proof extends to the equivariant case.

\begin{lemma}
Let $\phi : X \to Y$ be a $G$-equivariant morphism of Noetherian separated $G$-schemes with enough $G$-equivariant vector bundles and $T$ a $G$-invariant closed subset in $X$.
\begin{enumerate}
\item If $\phi|_T : T \to Y$ is proper of finite type and $S$ is a closed subset in $\phi(T)$ then $R \phi_*$ restricts to
\[ R \phi_* : \Dabs_T(\Coh^G(X),w \circ \phi) \to \Dabs_S(\Coh^G(Y),w). \]

\item Let $T_1, T_2$ be closed subsets of $X$. Then the tensor product restricts to a functor
\[ \otimes^L_X : \Dabs_{T_1}(\Coh^G(X),h_1) \times \Dabs_{T_2}(\Coh^G(X),h_2) \to \Dabs_{T_1 \cap T_2}(\Coh^G(X), h_1+h_2). \]

\item Let $S$ be a closed subset of $Y$. Then the pull-back restricts to a functor
\[ L \phi^* : \Dabs_S(\Coh^G(Y),h) \to \Dabs_{X \backslash \phi^{-1}(Y \backslash S)}(\Coh^G(X), h \circ \phi). \]
\end{enumerate}
\end{lemma}

\begin{proof}
1) By \cite[lemma 3.5]{EP} $R \phi_*$ restricts to a functor
\[ R \phi_* : \Dabs_T(\Coh^G(X),w \circ \phi) \to \Dabs(\Coh^G(Y),w)\]
so we only need to check the support. Since acyclicity is a local property we may assume that $X$ and $Y$ are affine so $\phi_*$ is exact. For $V \subseteq Y$ open $\phi_* \Mcal(V)=\Mcal(\phi^{-1}(V))$ so if $\Mcal|_U$ is absolutely acyclic then $\phi_* \Mcal|_V$ is absolutely acyclic for $V \subseteq \phi(U)$.

2) It is well-know that the derived tensor product of coherent sheaves on a Noetherian scheme is coherent. A tensor product is acyclic if one of the factors is acyclic and the other is flat. Thus, for an open set $\Vcal$ the matrix factorization $(\Mcal \otimes_{\Ocal_X} \Ncal)|_\Vcal =\Mcal|_\Vcal \otimes_{\Ocal_X |_{\Vcal}} \Ncal|_\Vcal$ is acyclic when $\Vcal \subseteq X\backslash T_1$ or $\Vcal \subseteq X\backslash T_2$. Hence if $\Vcal \subseteq X \backslash T_1 \cup X\backslash T_2=X \backslash (T_1 \cap T_2)$.

3) Let $\Mcal \in \Dabs_S(\Coh^G(Y),h)$. Then $\phi^{-1} \Mcal |_{\phi^{-1}(Y \backslash S)}$ is absolutely acyclic. By (2) this implies that $L \phi^*(\Mcal) \in \Dabs_{X \backslash \phi^{-1}(Y \backslash S)}(\Coh^G(X), h \circ \phi)$.
\end{proof}

\begin{cor}
\begin{enumerate}
\item The convolution action restricts to
\[* : \Dabs_{Y \times_V Y \times V^*}(\Coh^G(Y \times Y \times V^*),h) \times \Dabs(\Coh^G(Y \times X),w) \to \Dabs(\Coh^G(Y \times X),w)\]
\item The category $\Dabs_{Y \times_V Y \times V^*}(\Coh^G(Y \times Y \times V^*),h)$ is monoidal.
\end{enumerate}
\end{cor}
\begin{proof}
1) By the lemma the functors in the convolution all restrict to the full subcategory of coherent matrix factorizations. 
\[ \begin{tikzcd}
\Dabs_{Y \times_V Y \times V^*}(\Coh^G(Y \times Y \times V^*),h) \times \Dabs(\Coh^G(Y \times X),w) \arrow{d}{p^* \times p_{23}^*}\\
\Dabs_{Y \times_V Y \times X}(\Coh^G(Y \times Y \times X),h \circ p) \times \Dabs(\Coh^G(Y \times Y \times X),w \circ p_{23}) \arrow{d}{\otimes^L_{Y \times Y \times X}}\\
\Dabs_{Y \times_V Y \times X}(\Coh^G(Y \times Y \times X),w \circ p_{13}) \ar{d}{Rp_{13*}}\\
\Dabs(\Coh^G(Y \times X),w)
\end{tikzcd}\]

2) In the same way we get
\[ \begin{tikzcd}[row sep=0.2cm]
\Dabs_{Y \times_V Y \times V^*}(\Coh^G(Y \times Y \times V^*),h) \times \Dabs_{Y \times_V Y \times V^*}(\Coh^G(Y \times Y \times V^*),h) \arrow{ddd}{p_{12}^* \times p_{23}^*}\\
\\ \\
\Dabs_{Y \times_V Y \times Y \times V^*}(\Coh^G(Y \times Y \times Y \times V^*),h \circ p_{12}) \qquad \\
\qquad \times \Dabs_{Y \times Y \times_V Y \times V^*}(\Coh^G(Y \times Y \times Y \times V^*),h \circ p_{23})\arrow{ddd}{\otimes^L_{Y \times Y \times Y \times V^*}}\\
\\ \\
\Dabs_{Y \times_V Y \times_V \times Y \times V^*}(\Coh^G(Y \times Y \times Y \times V^*),h \circ p_{13}) \arrow{ddd}{Rp_{13*}}\\
\\ \\
\Dabs_{Y \times_V Y \times V^*}(\Coh^G(Y \times Y \times V^*),h). 
\end{tikzcd}\]

Since all categories are full subcategories of the ones involved in the convolution in the quasi-coherent setting the restricted convolutions are also associative.
\end{proof}

\section{Koszul duality}

The main theorem \ref{MainThm} would follow from \cite{BR} if we can construct a monoidal functor 
\[ D(\Coh^G(Y \times^R_V Y)) \to \Dabs_{Y \times_V Y \times V^*}(\Coh^G(Y \times Y \times V^*),h), \]
This turns out to be a bit hard so we will only construct the functor on a full subcategory which contains the generators of the braid group action. This functor is called Koszul duality, but it is not the functor from the previous chapter and no result from that chapter will be used. We prove all the properties of this functor needed for our proof.

\subsection{Definition of a Koszul duality functor}
Recall that $D(\Coh^G(Y \times^R_V Y))$ can be expressed as the normal derived category of a $DG$-category in the following way. Consider the function 
\[\rho : Y \times Y \to V, \qquad (y_1,y_2) \mapsto \nu(y_2)-\nu(y_1).\]
This induces a map $\rho^\sharp : V^* \to \Ocal_{Y \times Y}$. Since $Y \times_V Y=\rho^{-1}(0)$ we have a resolution
\[ \cdots \longrightarrow \Ocal_{Y \times Y} \otimes \bigwedge^2 V^* \longrightarrow \Ocal_{Y \times Y} \otimes V^* \stackrel{\theta}{\longrightarrow} \Ocal_{Y \times Y} \to \Ocal_{Y \times_V Y} \to 0, \]
where $\theta(f \otimes s)=f \rho^\sharp(s)$ and the differential is extended by Leibniz rule. Thus,
\[ D(\Coh^G(Y \times^R_V Y))\iso D(DG(\Ocal_{Y \times Y} \otimes \Lambda(V^*)-\module^G))\]

The first step in the construction is to define a functor
\[ DG(\Ocal_{Y \times Y} \otimes \Lambda(V^*)-\module^G) \to CDG(\Ocal_{Y \times Y} \otimes \Sym(V),h)-\module^G, \]
where $CDG(\Ocal_{Y \times Y} \otimes \Sym(V),h)-\module^G$ is the category of differential graded modules, where the differentials square to $h$ instead of zero. One way to construct such a functor is to tensor with a $\Ocal_{Y \times Y} \otimes \Lambda(V^*) \otimes \Sym(V)$-bimodule with differential $d$ satisfying $d^2=h$.

\begin{defin}
Pick a basis $(t_1, \dots, t_n)$ for $V$ and a dual basis $(\xi_1, \dots, \xi_n)$ for $V^*$. We define a grading with $\Ocal_{Y \times Y}$ in degree 0, $\xi_i$ in degree -1 and $t_i$ in degree 2. Consider the complex $K$ with terms
\[K^m:=\bigoplus_{m=2i-j} \Ocal_{Y \times Y}  \otimes \Lambda^j(\xi_1, \dots, \xi_n) \otimes \Sym^i(t_1, \dots, t_n)\]
and differential
\[ d(f \otimes x \otimes y):=f \otimes d_\Lambda(x) \otimes y+\sum_{k=1}^n f \otimes \xi_i x \otimes t_i y,\]
where $d_\Lambda$ is the usual differential on $\Lambda$. We call $K$ the Koszul complex.
\end{defin}

\begin{lemma}
The Koszul complex is in $CDG(\Ocal_{Y \times Y} \otimes \Sym(V),h)-\module^G$
\end{lemma}
\begin{proof}
We check that $d^2=h$.
\begin{align}
d^2(f \otimes x \otimes y) &=f \otimes d_\Lambda^2 x \otimes y +\sum_{k=1}^n f \otimes d_\Lambda(\xi_k x) \otimes t_k y+\sum_{k=1}^n f \otimes \xi_k(d_\Lambda x) \otimes t_k y\\
& \qquad \quad +\sum_{k,\ell=1}^n f \otimes \xi_k \xi_\ell x \otimes t_k t_\ell y\\
&=\sum_{k=1}^n f \otimes (d_\Lambda(\xi_k x)+\xi_k d_\Lambda x) \otimes t_k y\\
&=\sum_{k=1}^n f \otimes ( \rho^\sharp(\xi_k) x -\xi_k d_\Lambda x + \xi_k d_\Lambda x) \otimes t_k y\\
&=\sum_{k=1}^n f \otimes \rho^\sharp(\xi_k)x \otimes t_k y.
\end{align}
By definition we have $h : Y \times Y \times V^* \stackrel{\rho \times \id}{\longrightarrow} V \times V^* \stackrel{\langle \:,\: \rangle}{\longrightarrow} k$. Hence,
\begin{gather}
h^\sharp=(\rho^\sharp \otimes \id) \circ \langle \:,\: \rangle^\sharp=\sum_{k=1}^n \rho^\sharp(\xi_k) \otimes t_k.
\end{gather}
So $d^2=h$.
\end{proof}

Using the lemma we can define the functor
\begin{gather}
\kappa : DG(\Ocal_{Y \times Y} \otimes \Lambda(V^*)-\module^G) \to CDG(\Ocal_{Y \times Y} \otimes \Sym(V),h)-\module^G,\\
\Mcal \mapsto \Mcal \otimes_{\Ocal_{Y \times Y} \otimes \Lambda(V^*)} K \iso \Mcal \otimes \Sym(V).
\end{gather}
To make it $\Z/2$-graded we take the direct sum of the all the odd terms and all the even terms. Notice that
\[ (\Mcal \otimes_{\Lambda(V^*)} K)^m=\bigoplus_{m=i+2s-r} \Mcal^i \otimes_{\Lambda(V^*)} \Lambda^r(V^*) \otimes \Sym^s(V) \]
So $(\Mcal \otimes_{\Lambda(V^*)} K)^\text{odd} \iso \Mcal^\text{odd} \otimes \Sym(V)$ and likewise for the even part. 

\begin{lemma}
$\kappa$ descends to the homotopy categories.
\end{lemma}
\begin{proof}
Let $f : \Mcal \to \Ncal$ be a homotopy equivalence in $DG(\Ocal_{Y \times Y} \otimes \Lambda(V^*)-\module^G)$. We want to show that $\prod f \otimes \id : \kappa(\Mcal) \to \kappa(\Ncal)$ is a homotopy equivalence in $H^0(\QCoh^G(Y \times Y \times V^*),h)$, i.e. the diagram
\[ \xymatrix@=3pt{{\Mcal^\odd} \otimes \Sym(V) \ar[r]^{\prod f \otimes \id} \ar@<-0.5em>[d]_{d_{\kappa(\Mcal)}} & {\Ncal^\odd} \otimes \Sym(V) \ar@<-0.5em>[d]_{d_{\kappa(\Ncal)}} \\
{\Mcal^\even} \otimes \Sym(V) \ar@<-0.5em>[u]_{d_{\kappa(\Mcal)}} \ar[r]^{\prod f \otimes \id} & {\Ncal^\even} \otimes \Sym(V) \ar@<-0.5em>[u]_{d_{\kappa(\Ncal)}}
} \]
is commutative.
\begin{align}
(f \otimes \id) \circ d_{\kappa(\Mcal)}(m \otimes s)&=f (d_\Mcal m) \otimes s+\sum_{k=1}^n f(\xi_k m) \otimes t_k s\\
&=d_\Ncal f(m) \otimes s + \sum_{k=1}^n \xi_k f(m) \otimes t_k s\\
&=d_{\kappa(\Ncal)} \circ (f \otimes \id)(m \otimes s). \qedhere
\end{align}
\end{proof}

We want a functor into the absolute derived category of coherent sheaves so we cannot take infinite direct sums of coherent modules. To avoid this we restrict our functor to the following category.

\begin{defin}
Let $\Acal$ be a dg-scheme. The category $\Perf^G(\Acal)$ is the full subcategory of the dg-category of $G$-equivariant $\Acal$-dg-modules whose objects are finite complexes of locally free modules of finite rank.
\end{defin}

By the lemma we get a functor
\begin{gather}
\kappa : \operatorname{H}^0(\Perf^G(\Ocal_{Y \times Y} \otimes \Lambda(V^*))) \to \Dabs(\Coh^G(Y \times Y \times V^*),h)\\
\Mcal \mapsto (\Mcal^{\text{odd}} \otimes \Sym(V), \Mcal^{\text{even}} \otimes \Sym(V),d,d),\\
 d(m \otimes s):=d_\Mcal m \otimes s + \sum_{k=1}^n \xi_k m \otimes t_k s. 
\end{gather}

We want the functor to descend to the derived category ${D_{\Perf}} (\Ocal_{Y \times Y} \otimes \Lambda(V^*)-\module^G$). By lemma \ref{local} checking that something is acyclic can be done locally so we may assume that $Y$ is affine. 

\begin{lemma}
Assume that $Y$ is affine. Then there is an equivalence of triangulated categories
\[H^0(\Perf^G(\Ocal_{Y \times Y} \otimes \Lambda(V^*))) \iso {D_{\Perf}}(\Ocal_{Y \times Y} \otimes \Lambda(V^*)-\module^G).\]
\end{lemma}
\begin{proof}
When $Y$ is affine the categories reduce to categories of modules over a DG-algebra. Set $A:=\Ocal_{Y \times Y} \otimes \Lambda(V^*)$. The result would follow from showing that objects in $\Perf^G(A)$ are projective in $A-\module^G$. Equivalently, 
\[ \Ext^1_{A-\module^G}(P,M)=0 \]
for all $P \in \Perf^G(A)$ and $M \in A-\module^G$. Recall that
\[ \Ext^1_{A-\module^G}(P,M)=(\Ext^1_{A-\module}(P,M))^G.\]
The result now follows from the fact that perfect complexes are projective in $A-\module$.
\end{proof}

Thus, we have constructed a functor
\[\kappa: {D_{\Perf}}(\Ocal_{Y \times Y} \otimes \Lambda(V^*)-\module^G) \to \Dabs(\Coh^G(Y \times Y \times V^*),h).\]
However, the full subcategory ${D_{\Perf}}(\Ocal_{Y \times Y} \otimes \Lambda(V^*)-\module^G)$ is not preserved by convolution since the derived push-forward along a non-proper maps does not send $\Perf$ to $\Perf$. To fix this, we restrict to the full subcategory whose cohomology over $Y \times Y$ is set-theoretically supported on $Y \times_V Y$, so that the final projection is proper on the support.
\[\kappa: {D_{\Perf, Y \times_V Y}}(\Ocal_{Y \times Y} \otimes \Lambda(V^*)-\module^G) \to \Dabs(\Coh^G(Y \times Y \times V^*),h).\]

\subsection{Compatibility with convolution}

To prove that $\kappa$ commutes with convolution we need some preparatory lemmas.

\begin{lemma}
Let $\Mcal, \Ncal \in D_{\Perf}(\Ocal_{Y \times Y} \otimes \Lambda(V^*)-\module^G)$. Then $\kappa(\Mcal) \boxtimes \kappa(\Ncal) \iso \kappa(\Mcal \boxtimes \Ncal).$
\end{lemma}
\begin{proof}
The functor $\boxtimes$ is clearly exact and takes $\Perf \times \Perf$ to $\Perf$. First we check that the matrix factorizations agree on terms
{\allowdisplaybreaks
\begin{align}
&\kappa(\Mcal) \boxtimes \kappa(\Ncal)\\
&=\begin{pmatrix} \Mcal^\text{odd} \otimes \Sym(V) \boxtimes \Ncal^\text{even} \otimes \Sym(V) \oplus \Mcal^\text{even} \otimes \Sym(V) \boxtimes \Ncal^\text{odd} \otimes \Sym(V)  \\
\uparrow \quad \downarrow\\
    \Mcal^\text{odd} \otimes \Sym(V) \boxtimes \Ncal^\text{odd} \otimes \Sym(V) \oplus \Mcal^\text{even} \otimes \Sym(V) \boxtimes \Ncal^\text{even} \otimes \Sym(V) \end{pmatrix}\\
&= \begin{pmatrix}(\Mcal^\text{odd} \boxtimes \Ncal^\text{even}\oplus \Mcal^\text{even} \boxtimes \Ncal^\text{odd} )\otimes \Sym(V) \otimes \Sym(V) \\
\uparrow \quad \downarrow\\
(\Mcal^\text{odd} \boxtimes \Ncal^\text{odd}\oplus \Mcal^\text{even} \boxtimes \Ncal^\text{even} )\otimes \Sym(V) \otimes \Sym(V)  
\end{pmatrix}\\
&= \begin{pmatrix}(\Mcal \boxtimes \Ncal)^\text{odd} \otimes \Sym(V) \otimes \Sym(V) \\
\uparrow \quad \downarrow\\
(\Mcal \boxtimes \Ncal)^\text{even}\otimes \Sym(V) \otimes \Sym(V)
\end{pmatrix}\\
&=\kappa(\Mcal \boxtimes \Ncal).
\end{align}}
Now we check the differentials
{\allowdisplaybreaks
\begin{align}
&d_{\kappa(\Mcal) \boxtimes \kappa(\Ncal)}(a \otimes r \boxtimes b \otimes s) =d(a \otimes r) \boxtimes b \otimes s +(-1)^{|a|} a \otimes r \boxtimes d(b \otimes s)\\
& \enskip =(d_\Mcal a \otimes r +\sum_k \xi_k a \otimes t_k r) \boxtimes b \otimes s  +(-1)^{|a|} a \otimes r \boxtimes (d_\Ncal b \otimes s +\sum_k \xi_k b \otimes t_k s)\\
&d_{\kappa(\Mcal \boxtimes \Ncal)}(a \boxtimes b \otimes r \otimes s)\\
& \enskip=d(a \boxtimes b) \otimes r \otimes s +\sum_k (\xi_k,0) \cdot (a \boxtimes b) \otimes (t_k,0) \cdot (r \otimes t) \\
& \qquad \qquad +\sum_k (0,\xi_k) \cdot (a \boxtimes b) \otimes (0,t_k) \cdot (r \otimes s)\\
& \enskip =d_\Mcal a \boxtimes b \otimes r \otimes s +(-1)^{|a|} a \boxtimes d_\Ncal b \otimes r \otimes s \\
& \hspace{4em} +\sum_k (-1)^{|a| |0|} \xi_k a \boxtimes b \otimes t_k r \otimes s+\sum_k (-1)^{|a| |\xi_k |} a \boxtimes \xi_k b \otimes r \otimes t_k s\\
& \enskip =d_\Mcal a \boxtimes b \otimes r \otimes s +(-1)^{|a|} a \boxtimes d_\Ncal b \otimes r \otimes s \\
& \hspace{4em} +\sum_k (-1)^{|a|} \xi_k a \boxtimes b \otimes t_k r \otimes s+\sum_k a \boxtimes \xi_k b \otimes r \otimes t_k s \qedhere
\end{align}}
\end{proof}

\begin{lemma}
Let $\theta : Z_1 \to Z_2$ be a morphism of schemes. 
\begin{enumerate}
\item If $Z_3$ is a closed subscheme of $Z_1$ and $\theta$ restricted to $Z_3$ is proper then the following diagram is commutative
\[ \xymatrix{D_{\Perf}(\Ocal_{Z_2} \otimes \Lambda(V)-\module^G)  \ar[d]^{\kappa_1} & D_{\Perf,Z_3}(\Ocal_{Z_1} \otimes \Lambda(V)-\module^G) \ar[d]^{\kappa_2} \ar[l]_{(\theta \times \id)^{\sharp}_*} \\
\Dabs(\QCoh^G(Z_2 \times V),h_2) & \Dabs(\QCoh^G(Z_1 \times V),h_1) \ar[l]_{R(\theta \times \id)_*^f}} \]

\item If $\theta$ is flat then the following diagram is commutative
\[ \xymatrix{{D_{\Perf}}(\Ocal_{Z_2} \otimes \Lambda(V)-\module^G) \ar[r]^{(\theta \times \id)^{\sharp*}}  \ar[d]^{\kappa_1} & {D_{\Perf}}(\Ocal_{Z_1} \otimes \Lambda(V)-\module^G) \ar[d]^{\kappa_2}  \\
\Dabs(\QCoh^G(Z_2 \times V),h_2) \ar[r]^{(\theta \times \id)^*} & \Dabs(\QCoh^G(Z_1 \times V),h_1)} \]
\end{enumerate}
\end{lemma}

\begin{proof}
1) Since $\theta$ is proper on the support the functor $(\theta \times \id)^\sharp_*$ sends $\Perf$ to $\Perf$ so the composition is well-defined. Proving that $\kappa_1 (\theta \times \id)^\sharp_*(\Mcal) \iso (\theta \times \id)_* \kappa_2(\Mcal)$ can be done locally so we may assume that all schemes are affine in which case the push-forward is exact.
{\allowdisplaybreaks
\begin{align}
\kappa_2 (\theta \times \id)^\sharp_* (\Mcal) &= \left(\vcenter{\small{\xymatrix@R=1.5em{((\theta \times \id)^\sharp_* \Mcal)^{\text{odd}} \otimes \Sym(V^*) \ar@<1em>[d] \\
{((\theta \times \id)^\sharp_* \Mcal)^\even} \otimes \Sym(V^*) \ar@<1em>[u]
}}} \right)\\
& = \left(\vcenter{\small{\xymatrix@R=1.5em{(\theta \times \id)_* (\Mcal^{\text{odd}} \otimes \Sym(V^*)) \ar@<1em>[d] \\
(\theta \times \id)_* (\Mcal^\even \otimes \Sym(V^*)) \ar@<1em>[u]
}}} \right)\\
& \iso (\theta \times \id)_* \kappa_1(\Mcal).
\end{align}}
The functor $(\theta \times \id)^\sharp_*$ does not change the action of $\Lambda(V)$ so
\begin{align}
d_{\kappa_2(\theta \times \id)^\sharp_*(\Mcal)}(m \otimes s)&=dm \otimes s +\sum_{k=1}^n \xi_k m \otimes t_k s\\
&=d_{(\theta \times \id)_* \kappa_1(\Mcal)}(m \otimes s).
\end{align}

2) Pull-back preserve $\Perf$ and $(\theta \times \id)^*$ is exact so we have.
\begin{align}
\kappa_1 (\theta \times \id)^{\sharp *}(\Mcal)&=\left(\vcenter{\small{\xymatrix@R=1.5em{(\Mcal \otimes_{Z_2} \Ocal_{Z_1})^{\text{odd}} \otimes \Sym(V^*) \ar@<1em>[d]^d \\
(\Mcal \otimes_{Z_2} \Ocal_{Z_1})^{\text{even}} \otimes \Sym(V^*) \ar@<1em>[u]^d
}}} \right)\\
&\iso \left(\vcenter{\small{\xymatrix@R=1.5em{\Mcal^\text{odd} \otimes_{Z_2} \Ocal_{Z_1}\otimes \Sym(V^*) \ar@<1em>[d]^d \\
\Mcal^\text{even} \otimes_{Z_2} \Ocal_{Z_1} \otimes \Sym(V^*) \ar@<1em>[u]^d
}}} \right)\\
&\iso (\theta \times \id)^* \kappa_2(\Mcal).
\end{align}
For differentials we have
\begin{align}
d_{\kappa_1 (\theta \times \id)^{\sharp *}(\Mcal)}(m \otimes f \otimes s) &=d_\Mcal(m) \otimes f \otimes s +\sum_{k=1}^n \xi_k m \otimes f \otimes t_k s\\
&=d_{(\theta \times \id)^* \kappa_2(\Mcal)}(m \otimes f \otimes s).
\end{align}
This finishes the proof.
\end{proof}

\begin{lemma}
Let $f : V \hookrightarrow V \oplus W$ be an inclusion of vector spaces. Then the following diagram is commutative
\[ \xymatrix{{D_{\Perf}} (\Ocal_Z \otimes \Lambda(V) \otimes \Lambda(W)-\module^G) \ar[r]^-{(\id \times f)^\sharp_*} \ar[d]^{\kappa_1} & {D_{\Perf}} (\Ocal_Z \otimes \Lambda(V)-\module^G) \ar[d]^{\kappa_2} \\
\Dabs(\QCoh^G(Z \times V \times W),h_1) \ar[r]^-{(\id \times f)^*} & \Dabs(\QCoh^G(Z \times V),h_2) } \]
\end{lemma}
\begin{proof}
Since $f$ is injective and $\Lambda(W)$ is a finite complex with a finite dimensional vector space in each degree the functor $(\id \times f)^\sharp_*$ preserves $\Perf$. On the level of components we have
{\allowdisplaybreaks
\begin{align}
L(\id \times f)^* \kappa_1(\Mcal)&=\left(\vcenter{\small{\xymatrix@R=1.5em{\Mcal^\odd \otimes \Sym(V^*) \otimes \Sym(W^*) \otimes^L_{\Sym(W^*) \otimes \Sym(V^*)} \Sym(V^*) \ar@<1em>[d] \\
\Mcal^\even \otimes \Sym(V^*) \otimes \Sym(W^*) \otimes^L_{\Sym(W^*) \otimes \Sym(V^*)} \Sym(V^*)  \ar@<1em>[u]
}}} \right)\\
&\iso \left(\vcenter{\small{\xymatrix@R=1.5em{\Mcal^\odd \otimes \Sym(V^*) \ar@<1em>[d] \\
\Mcal^\even \otimes \Sym(V^*) \ar@<1em>[u]
}}} \right)\\
&=\kappa_2(\id \times f)^\sharp_*(\Mcal).
\end{align}}
Let $(\xi_k^V,t_k^V)_{k=1}^n$ be a pair of a basis for $V$ and its dual basis and $(\xi_l^W, t^W_l)_{l=1}^m$ a pair of a basis for $W$ and its dual basis. Then we have
{\allowdisplaybreaks
\begin{align}
&d_{L(\id \times f)^* \kappa_1(\Mcal)}(m \otimes 1 \otimes 1 \otimes s)\\
&=d_\Mcal m \otimes 1 \otimes 1 \otimes s + \sum_{k=1}^n \xi^V_k m \otimes 1 \otimes 1 \otimes t^V_k s + \sum_{l=1}^m \xi^W_l m \otimes 1 \otimes t^W_l \otimes s\\
&=d_\Mcal m \otimes 1 \otimes 1 \otimes s + \sum_{k=1}^n \xi^V_k m \otimes 1 \otimes 1 \otimes t^V_k s + \sum_{l=1}^m \xi^W_l m \otimes 1 \otimes1 \otimes  f^\sharp(t^W_l) s\\
&=d_\Mcal m \otimes 1 \otimes 1 \otimes s + \sum_{k=1}^n \xi^V_k m \otimes 1 \otimes 1 \otimes t^V_k s\\
&\iso d_\Mcal m \otimes s + \sum_{k=1}^n \xi_k^V m \otimes t^V_k s\\
&=d_{\kappa_2 (\id \times f)^\sharp_*(\Mcal)}(m \otimes s).
\end{align}}
This finishes the proof.
\end{proof}

\begin{prop}
The functor $\kappa$ is monoidal.
\end{prop}
\begin{proof}
Consider the derived projections
\[ \xymatrix{& Y\times_V^R Y \times_V^R Y \ar[dr]^{\bar{p}_{23}} \ar[d]^{\bar{p}_{13}} \ar[dl]_{\bar{p}_{12}} & \\ Y \times_V^R Y & Y \times_V^R Y & Y \times_V^R Y} \]
Recall that in \cite{BR} the convolution product for $\Mcal,\Ncal \in D_{\Perf, Y \times_V Y}(\Ocal_{Y \times Y} \otimes \Lambda(V^*)-\module^G)$ is defined as.
\[ \Mcal*\Ncal:=R \bar{p}_{13*}(L \bar{p}_{12}^* \otimes_{Y\times_V^R Y \times_V^R Y} L \bar{p}_{23}^*) \]

On the level of DG-schemes this translates into the following picture
\[ \xymatrix{ & \Ocal_{Y \times Y \times Y} \otimes \Lambda(V^*) \otimes \Lambda(V^*) & \\
\Ocal_{Y \times Y \times Y} \otimes \Lambda(V^*) \ar[ur]^{q_{12}} & \Ocal_{Y \times Y \times Y} \otimes \Lambda(V^*) \ar[u]^{q_{13}} & \Ocal_{Y \times Y \times Y} \otimes \Lambda(V^*) \ar[ul]_{q_{23}}\\
\Ocal_{Y \times Y} \otimes \Lambda(V^*) \ar[u]^{p_{12}^*} & \Ocal_{Y \times Y} \otimes \Lambda(V^*) \ar[u]^{p_{13}^*} & \Ocal_{Y \times Y} \otimes \Lambda(V^*) \ar[u]^{p_{23}^*}
} \]
Explicitly, the maps are given by
\begin{gather}
q_{12}(f \otimes v)=f \otimes v \otimes 1, \qquad q_{23}(f \otimes v)=f \otimes 1 \otimes v,\\
q_{13}(f \otimes v)=f \otimes v \otimes 1 + f \otimes 1 \otimes v.
\end{gather}
The map $p_{ij}^*$ is the pull-back corresponding to the projection to the $(i,j)$'th factor of $Y$ and identity on $V^*$. The formula becomes
\begin{align}
\Mcal * \Ncal&=R p_{13*} R q_{13*}(q_{12}^* p_{12}^* \Mcal \otimes^L_{\Ocal_{Y \times Y \times Y} \otimes \Lambda(V^*) \otimes \Lambda(V^*)} q_{23}^* p_{23}^* \Ncal)\\
&\iso R p_{13*} R q_{13*}((\Lambda(V^*) \otimes p_{12}^* \Mcal) \otimes^L_{\Ocal_{Y \times Y \times Y} \otimes \Lambda(V^*) \otimes \Lambda(V^*)} (p_{23}^* \Ncal \otimes \Lambda(V^*)))\\
&\iso R p_{13*} R q_{13*}(p_{12}^* \Mcal \otimes^L_{\Ocal_{Y \times Y \times Y}} p_{23}^* \Ncal)\\
&\iso R p_{13*} R q_{13*} L \Delta_Y^*(p_{12}^* \Mcal \boxtimes p_{23}^* \Ncal),
\end{align}
where $\Delta_Y$ is the diagonal embedding.

The corresponding diagram on the Koszul dual side is the following
\[ \xymatrix{ &Y \times Y \times Y \times V^* \times V^* & \\
Y \times Y \times Y \times V^* \ar@{^{(}->}[ur]^{\psi_{12}} \ar[d]^{\pi_{12}} & Y \times Y \times Y \times V^* \ar@{^{(}->}[u]^{\psi_{13}} \ar[d]^{\pi_{13}}  & Y \times Y \times Y \times V^* \ar@{_{(}->}[ul]_{\psi_{23}} \ar[d]^{\pi_{23}}\\
Y \times Y \times V^* & Y \times Y \times V^* & Y \times Y \times V^*
} \]
The morphisms in the Koszul dual picture are
\begin{gather}
\psi_{12}(y_1,y_2,y_3,v)=(y_1,y_2,y_3,v,0), \qquad \psi_{23}(y_1,y_2,y_3,v)=(y_1,y_2,y_3,0,v)\\
\psi_{13}(y_1,y_2,y_3,v)=(y_1,y_2,y_3,v,v).
\end{gather}
The $\pi_{ij}$ is projection to the $(i,j)$'th factor times identity. Using the above lemmas we get
\begin{align}
\kappa(\Mcal * \Ncal)&\iso R\pi_{13*} \kappa(R q_{13*} L \Delta_Y^*(p_{12}^* \Mcal \boxtimes p_{23}^* \Ncal))\\
&\iso R\pi_{13*} L\psi_{13}^* \kappa(L \Delta_Y^*(p_{12}^* \Mcal \boxtimes p_{23}^* \Ncal))\\
&\iso R\pi_{13*} L\psi_{13}^* L \Delta_Y^*\kappa(p_{12}^* \Mcal \boxtimes p_{23}^* \Ncal)\\
&\iso R\pi_{13*} L(\Delta_Y \circ \psi_{13})^*(\pi_{12}^* \kappa(\Mcal) \boxtimes \pi_{23}^* \kappa(\Ncal))\\
\end{align}
Notice that $\Delta_Y \circ \psi_{13} : Y \times Y \times Y \times V^* \to Y \times Y \times Y \times Y \times Y \times Y \times V^* \times V^*$ is the diagonal embedding $\Delta$ so
\begin{align}
\kappa(\Mcal * \Ncal)&\iso R\pi_{13*} L\Delta^*(\pi_{12}^* \kappa(\Mcal) \boxtimes \pi_{23}^* \kappa(\Ncal))\\
&\iso R\pi_{13*} (\pi_{12}^* \kappa(\Mcal) \otimes^L_{Y \times Y \times Y \times V^*} \pi_{23}^* \kappa(\Ncal))\\
&=\kappa(\Mcal) * \kappa(\Ncal).
\end{align}
This finishes the proof.
\end{proof}

\begin{prop}
The functor $\kappa$ takes the unit to the unit.
\end{prop}
\begin{proof}
The unit in $\Kcal_{Y,V}$ is the structure sheaf of the diagonal $\Ocal_{\Delta Y}$ sitting in degree 0.
\[ \kappa(\Ocal_{\Delta Y})=\left(\vcenter{\small{\xymatrix@=1em{0 \ar@<0.5em>[d] \\ \Ocal_{\Delta Y} \otimes \Sym(V) \ar@<0.5em>[u]}}}\right) \]
This is a push-forward along the inclusion 
\[i : Y \times_Y Y \times V^* \hookrightarrow Y \times Y \times V^*.\] 
Let $\Mcal \in \Dabs(\Coh^G(Y \times Y \times V^*),h)$. We need to check that $\kappa(\Ocal_{\Delta Y})*\Mcal \iso \Mcal \iso \Mcal * \kappa(\Ocal_{\Delta Y})$. Consider the following commutative diagram
\[ \xymatrix{ Y \times Y \times V^* \ar[d]^-{\id} \ar@{-}[r]^-\sim & Y \times_Y Y \times Y \times V^* \ar[d]^-{i_4} \ar[r]^-{p_1} & Y \times_Y Y \times V^* \iso Y \times V^*\ar[d]^i \\
Y \times Y \times V^* & Y \times Y \times Y \times V^* \ar[l]^-{p_{13}}_-{p_{23}} \ar[r]_-{p_{12}} & Y \times Y \times V^*} \]
Using flat base change and the projection formula we get
{\allowdisplaybreaks
\begin{align}
\kappa(\Ocal_{\Delta Y})*\Mcal &=R p_{13*}(p_{12}^* Ri_*\kappa(\Ocal_{\Delta Y}) \otimes^L_{Y \times Y \times Y \times V^*} p_{23}^*\Mcal)\\
&\iso R p_{13*}(Ri_{4*} p_1^*\kappa(\Ocal_{\Delta Y}) \otimes^L_{Y \times Y \times Y \times V^*} p_{23}^*\Mcal)\\
&\iso Rp_{13*} R i_{4*}(p_1^* \kappa(\Ocal_{\Delta Y}) \otimes^L_{Y \times_Y Y \times Y \times V^*} Li_4^* p_{23}^*\Mcal)\\
&\iso \id_* (p_1^* \kappa(\Ocal_{Y}) \otimes_{Y \times Y \times V* }^L \id^* \Mcal)\\
&\iso \left(\left(\vcenter{\small{\xymatrix@=1em{0 \ar@<0.5em>[d] \\ \Ocal_{Y \times Y \times V^*} \ar@<0.5em>[u]}}}\right) \otimes_{Y \times V^*} \left(\vcenter{\small{\xymatrix@=1em{0 \ar@<0.5em>[d] \\ \Ocal_{Y} \otimes \Sym(V) \ar@<0.5em>[u]}}}\right)\right) \otimes_{Y \times Y \times V^*}^L \Mcal\\
& \iso \left(\vcenter{\small{\xymatrix@=1em{0 \ar@<0.5em>[d] \\ \Ocal_{Y \times Y \times V^*} \ar@<0.5em>[u]}}}\right) \otimes^L_{Y \times Y \times V^*} \Mcal\\
&\iso \Mcal.
\end{align}}
Similarly, we get $\Mcal * \kappa(\Ocal_{\Delta Y}) \iso \Mcal$.
\end{proof}

We can now finish the proof of the main theorem.

\begin{proof}[Proof of main theorem \ref{MainThm}]
From theorem \ref{DGBraidAct} we have explicit generators of a braid group action in $\Kcal_{\gtil,\g}^{\Coh}$. These are sheaves on closed subschemes sitting in one degree so they lie in the full subcategory $D_{\Perf} (\Ocal_{\gtil \times \gtil} \otimes \Lambda(\g^*)-\module^G)$. Since all of them are supported on $\gtil \times_\g \gtil$ the images under the monoidal functor $\kappa$ land in the full subcategory $\Dabs_{\gtil \times_\g \gtil \times \g^*}(\Coh^G(\gtil \times \gtil \times \g^*),h)$. Thus, they act on $\Dabs(\Coh^G(\gtil \times T^* X),w)$ and generate the desired geometric braid group action.
\end{proof}

\end{document}